\documentclass[12pt]{revtex4-1}
\usepackage{amsmath,amssymb,times}
\usepackage{amsfonts,amsmath,amssymb,xcolor}
\usepackage{amsthm,amscd,bm}
\makeatletter

\@addtoreset{equation}{section}
\makeatother
\theoremstyle{definition}
\newtheorem{pr}{Proposition}[section]
\newtheorem{defn}{Definition}[section]

\newcommand{\bib}[2]{\frac{\partial {#1}}{\partial {#2}}}
\newcommand{\bbib}[3]{\frac{\partial^2 {#1}}{\partial {#2}{\partial {#3}}}}

\def\w{{\wedge}}

\begin{document}
\title{Finsler connection for general Lagrangian systems}
\author{L\'aszl\'o Kozma}
 \email{kozma@unideb.hu}
 \affiliation{Insitute of Mathematics, University of Debrecen, P O Box 12, H-4010 Debrecen, Hungary}
\author{Takayoshi Ootsuka}
 \email{ootsuka@cosmos.phys.ocha.ac.jp}
 \affiliation{Physics Department, Ochanomizu University, 2-1-1 Ootsuka Bunkyo Tokyo, Japan }
\date{\today}

\begin{abstract}

We give a new simplified definition of a non-linear connection of Finsler geometry
which could be applied for not only regular case but also singular case.
For the regular case, it corresponds to the non-linear part of
the Berwald's connection,
but our connection is expressed not in the line element space but in the point-Finsler space.
In this view we recognize a Finsler metric $L(x,dx)$ as a ``non-linear form'',
which could be regarded as a generalisation
of the original expression of Riemannian metric, $\sqrt{g_{\mu\nu}(x)dx^\mu dx^\nu}$.
Furthermore our formulae are easy to calculate compared to the conventional methods,
which encourages the application to physics.
This definition can be used in the case where the Finsler metric is singular,
which corresponds to gauge constrained systems in mechanics.
Some non-trivial examples of constrained systems
are introduced for exposition of applicability of the connection.

\end{abstract}

\maketitle

\section{Introduction}

Finsler geometry has a large potential of applications to physics or other mathematical
sciences.
Usually a Finsler metric $L(x,dx)$ on $M$ is defined by a function of $TM$ which satisfies
i) {\it regularity},
ii) {\it positive homogeneity}, and iii) {\it strong convexity},
in the standard textbook~\cite{Bao-Chern-Shen, ChernChenLam}.
However, in applications of Finsler geometry to Lagrange systems,
almost all systems do not satisfy the regularity condition, in other words,
$L$ is not defined on the whole slit tangent bundle $TM^{\circ}=TM-\{0\}$
but only on a sub-bundle $D(L) \subset TM$, depending on $L$.
Furthermore, the most important Lagrangian systems in physics: the gauge theories,
do not satisfy the strong convexity.
Therefore we will only consider the weaker regularity condition and the
following positive homogeneity condition of $L$ as the definition of Finsler metric.
\begin{eqnarray}
 L(x,\lambda dx)=\lambda L(x,dx), \quad {}^\forall \lambda >0.\label{homo}
\end{eqnarray}

Any Lagrangian systems of finite degree of freedom can be reformulated
in such Finsler manifolds~\cite{Lanczos, Suzuki1956} without changing their physical
contents,
and the action functional is given by the integral of the Finsler metric
which is made from the Lagrangian,
then the variational principle becomes geometric and
independent of parameterisation, which we will call covariant.

From the point of view of a physicist,
especially when thinking about the Lagrangian formulation,
we are inclined to define a non-linear connection not on a line element space $TM^\circ$
but directly on the point manifold $M$~\cite{Kozma-Tamassy1}.
Usual treatments of Finsler connection based on line element space; slit tangent bundle
or projected tangent bundle; are rather similar to the
Hamiltonian formulation.
The best covariant Hamiltonian formulation
using contact manifold~\cite{Aldaya} is deeply related to the covariant Lagrangian formulation.
Furthermore, the Hamiltonian formulation corresponds to
projected tangent bundle formulation of Finsler geometry~\cite{ChernChenLam}
in special cases.
Since we would like to consider Lagrangian formulation and not Hamiltonian formulation,
we think that the point Finsler viewpoint is better suited for our objective.
If we consider only Euler-Lagrange equation, symmetry of the system
or Noether's theorem~\cite{Ootsuka2012, OYIT2014},
Finsler connection is not a necessity.
However, if we want to consider the auto-parallel forms
of Euler-Lagrange equation, or
seek new symmetries and conserved quantities,
our Finsler non-linear connection can be a huge help to us.

In the next section we give a generalisation of linear connection of a vector bundle
$E \stackrel{\pi}{\to} M$
to a non-linear connection, from a some different kind of view.
We firstly define a non-linear connection not to vector fields
but to dual (covector) fields,
by generalising the coefficients of linear connection as functions of $x^\mu$
and $e^a$, where $x^\mu$ are coordinates functions of $M$ and $e^a$ are the dual basis of
$e_a$, the basis of section of $E$.
Finally we will define the non-linear covariant derivative of $\Gamma(E)$
by the covariant derivative of $\Gamma(E^\ast)$ from using their duality.
In Section 3, we generalise a linear connection of $TM$ to the non-linear connection
preserving Finsler metric $L$, that is our Finsler non-linear connection.
In our point of view, the non-linear connection is defined on $\Gamma(TM)$,
that is tangent vector fields over $M$, and leads to a non-linearly parallel transport
preserving Finsler norm.
There we show the existence and uniqueness of such a non-linear connection
of a general Finsler metric including singular metric.
In Section 4, we give a short review of a covariant Lagrangian formulation
using Finsler manifold, and give the Euler-Lagrange equation to an auto-parallel
form in the general cases.
In the last section, we give some examples of Lagrangian systems
including non-trivial gauge constrained systems.
We hope that our non-linear connection will be  applied to more several areas.

\section{Non-linear generalisation of linear connection}

Let $\stackrel{\circ}{\nabla}$ be a linear connection on a vector bundle
$E \stackrel{\pi}{\to} M$, i.e.,
$\stackrel{\circ}{\nabla}: \Gamma(E) \rightarrow \Gamma(T^* M \otimes E)$.
The contraction with the vector field $X=X^\mu \bib{}{x^\mu}$ over $M$
gives the covariant derivative $\stackrel{\circ}{\nabla}_X:\Gamma(E)\rightarrow \Gamma(E)$
along $X$.
This $\stackrel{\circ}{\nabla}$ can also define the covariant derivative on
dual vector bundle $E^\ast$,
$\stackrel{\circ}{\nabla}_X:\Gamma(E^\ast)\rightarrow \Gamma(E^\ast)$.
In coordinates, this can be given by
$\stackrel{\circ}{\nabla}_{\bib{}{x^\mu}}\!e^a
= -
{\stackrel{\circ}{\varGamma}} {{}^a}_{b \mu}(x)e^b
$,
where $e^\ast=\{e^a\}$ is the local basis of $E^\ast$.
We generalise this linear connection
$\stackrel{\circ}{\nabla}$ on vector bundle
to non-linear connection $\nabla$
by replacing
${\stackrel{\circ}{\varGamma}} {{}^a}_{b\mu}(x)e^b$
by ${\varGamma^a}_{\mu}(x, e^*)$,
which is a $1$-degree homogeneous function of $e^c$.
That is, the non-linear connection $\nabla$ is defined by
\begin{eqnarray}
\nabla e^a := -{\varGamma^a}(x,e^\ast)
 =- dx^\mu \otimes {\varGamma^a}_{\mu}(x,e^\ast), \quad
\nabla_{\bib{}{x^\mu}}e^a = - {\varGamma^a}_{\mu}(x,e^\ast).
\end{eqnarray}
For a general section $\rho=\rho_a e^a \in \Gamma(E^\ast)$,
\begin{eqnarray}
 \nabla \rho := d\rho_a \otimes e^a
 - \rho_a dx^\mu \otimes {\varGamma^a}_{\mu}(x,e^\ast).
\end{eqnarray}
Notice that $\nabla_X \rho \not \in \Gamma({E^\ast}) $, however
it holds linearity,
$\nabla_X (\rho_1 + \rho_2) = \nabla_X \rho_1 + \nabla_X \rho_2$.

The action of $\nabla$ to the section of $E$ can be defined as follows.
Let $\xi=\xi^a \otimes e_a$ be a smooth section of $E$, and consider the derivative of
$\xi^a=\langle e^a,\xi\rangle$,
\begin{eqnarray}
 d\xi^a=\langle \nabla e^a,\xi\rangle+\langle e^a,\nabla \xi\rangle
 =-\langle {\varGamma^a}(x,e^\ast),\xi\rangle+\langle e^a,\nabla \xi\rangle,
\end{eqnarray}
therefore we define
\begin{eqnarray}
 \nabla \xi:=
 \left(d\xi^a+{\varGamma^a}(x,e^\ast(\xi)) dx^\mu \right)\otimes e_a, \quad
 \nabla_X \xi= X^\mu \left( \frac{\partial \xi^a}{\partial x^\mu}
+{\varGamma^a}_{\mu}(x,e^\ast(\xi)) \right)\otimes e_a.
\end{eqnarray}
$\nabla_X$ maps $\Gamma(E)$ to $\Gamma(E)$, but the linearity does not hold;
$\nabla_{X}(\xi_1+\xi_2) \neq \nabla_{X} \xi_1 + \nabla_X \xi_2$.
Taking a coordinates $(x^\mu, e^a)$ of $E$,
we can define a distribution ${\cal N}$ of $TE$:
\begin{eqnarray}
 {\cal N}=
 \left\langle \frac{\delta}{\delta x^\mu}:=\bib{}{x^\mu}-{\varGamma^a}_\mu(x,e^\ast)
 \bib{}{e^a} 
 \right\rangle,
\end{eqnarray}
which is a usual definition of non-linear connection of $E$~\cite{Miron}.

For physical problems, we use more often the derivative of covariant quantities
than contravariant quantities, so this definition of non-linear connection
is useful to applications of physics.

\section{Generalised Berwald's non-linear connection}

Let us consider a Finsler manifold $(M,L)$,
where $M$ is a $(n+1)$-dimensional differentiable manifold and
$L(x^\mu,dx^\mu)$ be a Finsler metric, using coordinates $(x^\mu)$ of $M$,
which is  a function of $x^\mu$ and $dx^\mu$ and
1-degree homogeneous function of $dx^\mu$.
In our introduction we assume only regurality on a sub-bundle $D(L) \subset TM$ and
homogeneity condition (\ref{homo}), and not the convexity condition.
We define a new Finsler connection,
which is a generalisation of Berwald's connection,
using the previous non-linear connection
defined on $TM$ which preserves the Finsler 1-form $L$;
$\nabla L=0$.
\begin{defn}
A non-linear Finsler connection $\nabla$ is such that satisfies the following conditions:
\begin{eqnarray}
& \displaystyle
\nabla dx^\alpha
 =-{N^\alpha}(x,dx)
 =-dx^\mu \otimes {N^\alpha}_{\mu}(x,dx), \label{N}
\\
& \displaystyle
\bib{{N^\alpha}_\mu}{dx^\beta}-\bib{{N^\alpha}_\beta}{dx^\mu}=0, \label{symmetry}\\
& \displaystyle
\bib{L}{x^\mu}=p_\alpha {N^\alpha}_{\mu}, \quad
 p_\alpha:=\bib{L}{dx^\alpha},
\label{Fpreserve}
\end{eqnarray}
where ${N^\alpha}_{\mu}={N^\alpha}_{\mu}(x,dx)$ and
${N^\alpha}_{\beta\mu}:=\bib{{N^\alpha}_\mu}{dx^\beta}$ are $1$-degree and $0$-degree
homogeneous functions of $dx=\{dx^\mu\}$ each other.
\end{defn}
The last condition means the condition of preserving the Finsler metric $L$;
$0=\nabla L=\nabla x^\mu \bib{L}{x^\mu}+\nabla dx^\mu \bib{L}{dx^\mu}$,
which is a generalisation of covariant derivative to ``non-linear form'' $L(x,dx)$.
The covariant derivative of tangent vector field $Z=Z^\mu \bib{}{x^\mu}$ on $M$ by this 
non-linear connection can be defined same as previous section:
\begin{eqnarray}
 \nabla Z:=\left\{dZ^\mu +{N^\mu}_\alpha \left(x,dx(Z)\right) dx^\alpha \right\}
 \otimes \bib{}{x^\mu},
\end{eqnarray}
and the covariant derivative of $Z$ along a tangent vector field $X=X^\mu \bib{}{x^\mu}$
on $M$ also can be defined by,
\begin{eqnarray}
 \nabla_X Z:=X^\nu \left\{\bib{Z^\mu}{x^\nu}+{N^\mu}_{\nu}\left(x,dx(Z)\right)\right\}
 \otimes \bib{}{x^\mu}.
\end{eqnarray}
Here we consider only tangent vectors only on the point manifold $M$ not on 
the line element space.
\begin{pr}
The Finsler norm of vectors is conserved by a parallel displacement along a curve
$c(t)$.
That is $\displaystyle \frac{dL(Z)}{dt}=0$ if $\nabla_{\dot{c}(t)} Z=0$.
\end{pr}
\begin{proof}
\begin{eqnarray}
 \frac{dL(Z)(c)}{dt}
 =\frac{dx^\mu(c)}{dt} \bib{L}{x^\mu}+\frac{dx^\nu(c)}{dt}\bib{Z^\mu}{x^\nu} \bib{L}{dx^\mu}
 =\frac{dx^\nu(c)}{dt}
 \left\{
 {N^\mu}_\nu\left(x,dx(Z)\right)+\bib{Z^\mu}{x^\nu} 
 \right\}\bib{L}{dx^\mu}=0. \nonumber
\end{eqnarray} 
\end{proof}
\begin{defn}
We denote $L_{\mu \nu}:= \bbib{L}{dx^\mu}{dx^\nu}$.
For a Finsler metric of our definition,
$(n+1)\times (n+1)$ matrix $(L_{\mu\nu})$ always satisfies
${\rm rank}\left(L_{\mu \nu}\right)\leq n$, where we assume ${\rm dim}M=n+1$.
If ${\rm rank}\left(L_{\mu \nu}\right)=n$,
then the Finsler metric $L$ is called {\it regular}.
Otherwise it is  called {\it singular}.
\end{defn}
In applications to Lagrangian systems, we will see later
that the Lagrangian of non-constrained systems correspond to regular Finsler
metrics, and constrained systems (gauge theories) correspond
to singular Finsler metrics.
In application to physics, the singular Finsler manifolds are very important
for gauge theories which often appear in several areas of physics.

If $L$ is a singular Finsler metric and $(L_{\mu\nu})$ has constant rank:
${\rm rank}\left(L_{\mu \nu}\right)=n-1-D$,
there are $D$ independent zero eigenvectors
$v^\mu_I(x,dx) \, (I=1,2,\cdots ,D)$ which are functions of $x^\mu$ and $dx^\mu$ and,
satisfying
\begin{eqnarray}
 L_{\mu\nu}v^\nu_I=0, \quad p_\mu v^\mu_I=0. \label{vI}
\end{eqnarray}
If $v^\mu_I$ satisfy only the former and not the latter $p_\mu v^\mu_I=w_I\neq 0$,
we can replace $v^\mu_I$ to $v^\mu_I-w_I \frac{dx^\mu}{L}$.

\begin{pr}
Take a coordinate system such that ${\rm det}(L_{ab})\neq 0, \,
(a,b=D+1,D+2,\cdots n)$ then,
\begin{eqnarray}
 \begin{array}{lll}
 \displaystyle
 \ell^\mu_0=\frac{dx^\mu}{L}, \quad &
 \ell^\mu_I= v^\mu_I, &  
 \displaystyle
 \ell^\mu_a=L\bib{\ell^\mu_0}{dx^a}=\delta^\mu_a-\frac{p_a dx^\mu}{L},
 \\
 & (I=1,2,\cdots D), \,& (a=D+1,D+2,\cdots, n),
 \end{array}
\end{eqnarray}
become $(n+1)$ independent vectors.
\end{pr}

\begin{proof}
Suppose $A \ell^\mu_0+B^I \ell^\mu_I+C^a \ell^\mu_a=0$.
However, multiplying this by $p_\mu$ gives,
\begin{eqnarray}
 A=0, \quad B^I \ell^\mu_I+C^a \ell^\mu_a=0,
\end{eqnarray}
since $p_\mu \ell^\mu_I=p_\mu \ell^\mu_a=0$.
Differentiate the second equation by $dx^b \, (b=D+1,D+2,\cdots,n)$ and multiply by $p_\mu$,
\begin{eqnarray}
 B^Ip_\mu \bib{\ell^\mu_I}{dx^b}
 + C^a p_\mu \bib{\ell^\mu_a}{dx^b}=0,
\end{eqnarray}
and we use the formula,
\begin{eqnarray}
 L_{\mu b}v^\mu_I+p_{\mu}\bib{v^\mu_I}{dx^b}=p_\mu \bib{\ell^\mu_I}{dx^b}=0,
\end{eqnarray}
obtained by differentiating the latter equation of (\ref{vI}) by $dx^b$.
We get,
\begin{eqnarray}
 C^a p_\mu \bib{\ell^\mu_a}{dx^b}
 =C^a p_\mu \left(-L_{ba}\frac{dx^\mu}{L}-p_a \frac{\ell^\mu_b}{L} \right)
 =-L_{ba}C^a=0,
\end{eqnarray}
so $C^a=0$ and $B_I \ell^\mu_I=0$.
Finally since $\ell^\mu_I=v^\mu_I$ are independent vectors,
we get $B_I=0$.
\end{proof}

It is very helpful for calculation
if we define an auxiliary function
$G^\mu:=\frac12 {N^\mu}_{\alpha}dx^\alpha$.
Straight forward calculation with the use of the homogeneity property of 
${N^\mu}_{\alpha}$
and (\ref{symmetry}) gives us,
\begin{eqnarray}
 &&
 \bib{G^\mu}{dx^\alpha}=\frac12{N^\mu}_{\alpha}
 +\frac12\bib{{N^\mu}_{\beta}}{dx^\alpha}dx^\beta
 ={N^\mu}_{\alpha}, \quad
 \frac{\partial^2 G^\mu}{\partial dx^\beta \partial dx^\alpha}
 =\bib{{N^\mu}_{\alpha}}{dx^\beta}
 =:{N^\mu}_{\alpha\beta}.
\end{eqnarray}
Therefore we are able to express the coefficients of connection
by $G^\mu$:
\begin{eqnarray}
 {N^\mu}_{\alpha}=\bib{G^\mu}{dx^\alpha},
 \qquad
 {N^\mu}_{\alpha\beta}
 = \frac{\partial^2 G^\mu}{\partial dx^\beta \partial dx^\alpha}.
 \label{G}
\end{eqnarray}
To determine the connection, it is sufficient to consider
$G^\mu$ instead of ${N^\mu}_{\alpha}$.

\begin{pr}[Existence] \label{exist}
Let $L$ be a singular Finsler metric with ${\rm det}(L_{ab})\neq 0 \, (a,b=D+1,D+2,\dots,n)$.
There exists a non-linear Finsler connection such that their coefficients are expressed
by the following $G^\mu$;
 \begin{eqnarray}
 &&
 G^\mu=\frac12
 \left( dx^\beta \bib{L}{x^\beta}\right) \ell^\mu_0
 +\lambda^I \ell^\mu_I+\lambda^a \ell^\mu_a, \quad
 M_I=L^{ab}L_{aI} M_b, \label{solG}\\
 && \lambda^a=L^{ab}M_b,\quad
 M_\mu=\frac12 \left(- \bib{L}{x^\mu}+
dx^\rho \frac{\partial^2 L}{\partial dx^\mu \partial x^\rho}\right).
\end{eqnarray}
Where $\ell^\mu_0=\frac{dx^\mu}{L}$, $\ell^\mu_I=v^\mu_I$,
$\ell^\mu_a=L\bib{\ell^\mu}{dx^i}$,
$\mu,\beta=0,1,2,\dots ,n$,
$I=1,2,\dots, D$, $a, b=D+1,D+2,\dots,n$,
$L^{ab}$ is the inverse of $L_{ab}$ and $\lambda^I$ are arbitrary
function.
\end{pr}

\begin{proof}
From the definition
of coefficients of the connection (\ref{G}) and
(\ref{symmetry}) are automatically satisfied,
so it only remains to show (\ref{Fpreserve}).
If we multiply (\ref{Fpreserve}) by $dx^\beta$, we get
inhomogeneous linear equation of $G^\mu$;
\begin{eqnarray}
 dx^\beta \bib{L}{x^\beta}=2p_\mu G^\mu. \label{omega}
\end{eqnarray}
This can be solved as
\begin{eqnarray}
 G^\mu=
 \frac12 \left(dx^\beta \bib{L}{x^\beta} \right) \frac{dx^\mu}{L}
 +\lambda^I \ell^\mu_I+\lambda^a \ell^\mu_a,
\end{eqnarray}
using  $\ell^\mu_0, \ell^\mu_I, \ell^\mu_a$
and $p_\mu \ell^\mu_0=1$, $p_\mu \ell^\mu_I=p_\mu \ell^\mu_a=0$.
Here $\lambda^I$ and $\lambda^a$ are still unknown functions of $x^\mu$ and $dx^\mu$.
Let us determine $\lambda^a$ from (\ref{Fpreserve}).
Differentiating $G^\mu$ by $dx^\beta$,
\begin{eqnarray}
 \bib{G^\mu}{dx^\beta}=\frac12\left(
			       \bib{L}{x^\beta}
 +dx^\nu \frac{\partial^2 L}{\partial dx^\beta \partial x^\nu}
 \right) \frac{dx^\mu}{L}
 +\frac12 \left(dx^\gamma \bib{L}{x^\gamma} \right)
 \frac{L\delta^\mu_\beta-p_\beta dx^\mu}{L^2} \nonumber
 \\
 +\bib{\lambda^I}{dx^\beta}\ell^\mu_I
 +\lambda^I \bib{\ell^\mu_I}{dx^\beta}
 +\bib{\lambda^a}{dx^\beta}\ell^\mu_a
 +\lambda^a \bib{\ell^\mu_a}{dx^\beta},
\end{eqnarray}
and putting this into the right hand side of (\ref{Fpreserve}),
\begin{eqnarray}
 p_\mu {N^\mu}_{\alpha\beta}dx^\alpha
 =p_\mu \bib{G^\mu}{dx^\beta}
 =\frac12\left(
 \bib{L}{x^\beta}
 +dx^\nu \frac{\partial^2 L}{\partial dx^\beta \partial x^\nu}
 \right)+\lambda^a p_\mu \bib{\ell^\mu_a}{dx^\beta}. \label{pGamma}
\end{eqnarray}
The last term of (\ref{pGamma}) becomes
\begin{eqnarray}
 p_\mu \bib{\ell^\mu_a}{dx^\beta}
 =p_\mu \left(-\frac{\delta^\mu_\beta p_a}{L}+\frac{dx^\mu p_a p_\beta}{L^2}
 -\frac{dx^\mu}{L}\bib{p_a}{dx^\beta}
 \right)
 =-\bib{p_a}{dx^\beta}
 =-L_{\beta a}, 
\end{eqnarray}
then (\ref{Fpreserve}) becomes the following equations for $\lambda^i$;
\begin{eqnarray}
 L_{\beta a}\lambda^a
 =M_\beta, \quad
 M_\beta:=\frac12 \left(
 -\bib{L}{x^\beta}+dx^\nu \frac{\partial^2 L}{\partial dx^\beta \partial x^\nu}
 \right).
 \label{lambda}
\end{eqnarray}
First, we can solve the following $(n-D)$-linear equations of the (\ref{lambda}),
\begin{eqnarray}
 L_{ab}\lambda^a=M_b.
 \label{lambda1}
\end{eqnarray}
From (\ref{lambda1}) we can determine $\lambda^a$
by using the inverse matrix $L^{ab}$,
\begin{eqnarray}
 \lambda^a=L^{ab}M_b. \label{lambdaans}
\end{eqnarray}
The other equations which can be obtained from (\ref{lambda}) are,
\begin{eqnarray}
 L_{0a}\lambda^a=M_0, \quad
 L_{Ia} \lambda^a=M_I. \label{lambda2}
\end{eqnarray}
The second equation of (\ref{lambda2}) implies the first, since
\begin{eqnarray}
 dx^0 L_{0a}\lambda^a
 &=& \left(-dx^I L_{Ia}-dx^b L_{ba}\right) \lambda^a
 = -dx^I M_I-dx^b M_b =-dx^\mu M_\mu+dx^0 M_0 \nonumber \\
 &=& -dx^\mu \frac12 \left(
 -\bib{L}{x^\mu}+dx^\beta \frac{\partial^2 L}{\partial dx^\mu \partial x^\beta}
 \right)+dx^0 M_0 \nonumber \\
  &=&\frac12\left(
 dx^\mu \bib{L}{x^\mu}-dx^\beta \bib{L}{x^\beta}
 \right)+dx^0 M_0= dx^0 M_0.
\end{eqnarray}
The second equation of (\ref{lambda2}) should be regarded as constraints of the system.
This derivation shows that $G^\mu$ exists,
and is uniquely determined up to arbitrary $\lambda^I \, (I=1,2,\dots,D)$.
\end{proof}

\begin{pr}[Uniqueness]\label{unique}
 Up to arbitrary $D$ functions $\lambda^I \, (I=1,2,\dots, D)$,
 non-linear Finsler connection ${N^\mu}_{\alpha}(x,dx)$
 which satisfies (\ref{symmetry}) and (\ref{Fpreserve})
 is uniquely obtained by
 $\displaystyle {N^\mu}_{\alpha}=\bib{G^\mu}{dx^\alpha}$,
 where $G^\mu$ is of the previous proposition.
\end{pr}
\begin{proof}
 From {Proposition \ref{exist}},
 $G^\mu$ is unique. So we should prove that
 if ${N^\mu}_{\alpha}$ and ${\tilde{N}^\mu}{}_{\alpha}$
 satisfy (\ref{symmetry}), (\ref{Fpreserve})
 and $G^\mu=\frac12{N^\mu}_{\alpha}dx^\alpha
 =\frac12{\tilde{N}^\mu}{}_{\alpha}dx^\alpha$,
 then ${N^\mu}_{\alpha}={\tilde{N}^\mu}{}_{\alpha}$.
 \\
 Let us define ${B^\mu}_{\alpha}={\tilde{N}^\mu}{}_{\alpha}-{N^\mu}_{\alpha}$.
 Then ${B^\mu}_{\alpha}$ satisfies
\begin{eqnarray}
 \bib{{B^\mu}_{\alpha}}{dx^\beta}=\bib{{B^\mu}_{\beta}}{dx^\alpha}, \quad
 \bib{{B^\mu}_{\alpha}}{dx^\beta}dx^\beta
 =\bib{{B^\mu}_{\beta}}{dx^\alpha}dx^\beta={B^\mu}_\alpha, \quad
 {B^\mu}_\alpha dx^\alpha=0. \label{B3}
\end{eqnarray}
Differentiating the third equation of (\ref{B3}) with respect to $dx^\beta,$
\begin{eqnarray}
 0=\bib{{B^\mu}_{\alpha}}{dx^\beta} dx^\alpha
     +{B^\mu}_{\beta}
 =2{B^\mu}_{\beta}.
\end{eqnarray}
\end{proof}

\section{Euler-Lagrange equation}

For an arbitrary singular Finsler manifold $(M,L)$ which has ${\rm rank}(L_{\mu\nu})=n-D$,
the Euler-Lagrange equations are defined by
\begin{eqnarray}
 0=\bib{L}{x^\alpha}-d\left(\bib{L}{dx^\alpha}\right)
 =\bib{L}{x^\alpha}-
\frac{\partial^2 L}{\partial x^\beta \partial dx^\alpha}dx^\beta
-\frac{\partial^2 L}{\partial dx^\beta \partial dx^\alpha}d^2 x^\beta.
\label{EL}
\end{eqnarray}
Precisely (\ref{EL}) are equations for an oriented curve
$\bm{c} \subset M$:
\begin{eqnarray}
 0=c^\ast\left\{\bib{L}{x^\mu}-d\left(\bib{L}{dx^\mu}\right)\right\},
\end{eqnarray}
by using a parameterisation of $\bm{c}$, i.e.~
$c(t):T \subset \mathbb{R} \to M$,
and $dx^\mu$ and $d^2x^\mu$ pull-backed:
\begin{eqnarray}
 c^\ast dx^\mu=\frac{dx^\mu(t)}{dt}dt, \quad
 c^\ast d^2 x^\mu=\frac{d^2 x^\mu(t)}{dt^2} dt^2. 
\end{eqnarray}
In our paper for avoiding cumbersome symbol $c^\ast$, we will drop the pull-back symbols.
Please look at the paper~\cite{OYIT2014} for these convenient notations in more details.

Let us express this equation with the previous non-linear Finsler connection.
From the condition (\ref{Fpreserve}),
by differentiation by $dx^\alpha$, we can get
\begin{eqnarray}
 \frac{\partial^2 L}{\partial dx^\alpha \partial x^\beta}
 =\bib{p_\mu}{dx^\alpha} {N^\mu}_{\beta}
 +p_\mu \bib{{N^\mu}_{\beta}}{dx^\alpha}.
\end{eqnarray}
Multiplying this by $dx^\beta$,
\begin{eqnarray}
 \frac{\partial^2 L}{\partial dx^\alpha \partial x^\beta}dx^\beta
=\bib{p_\mu}{dx^\alpha} {N^\mu}_{\beta} dx^\beta
+p_\mu \bib{{N^\mu}_{\beta}}{dx^\alpha} dx^\beta. 
\end{eqnarray}
Furthermore with (\ref{symmetry}) and
$1$-st homogeneity of ${N^\mu}_{\beta}$,
\begin{eqnarray}
 \frac{\partial^2 L}{\partial x^\beta \partial dx^\alpha}dx^\beta
=\frac{\partial^2 L}{\partial dx^\alpha \partial dx^\mu}
{N^\mu}_{\beta}dx^\beta+\bib{L}{x^\alpha}. 
\end{eqnarray}
Then
\begin{eqnarray}
 \bib{L}{x^\alpha}-
 \frac{\partial^2 L}{\partial x^\beta \partial dx^\alpha}dx^\beta
 =-\frac{\partial^2 L}{\partial dx^\alpha \partial dx^\mu}
{N^\mu}_{\beta}dx^\beta. 
\end{eqnarray}
Therefore Euler-Lagrange equation (\ref{EL}) can be written as
\begin{eqnarray}
 0=\frac{\partial^2 L}{\partial dx^\alpha \partial dx^\mu}
\left( d^2x^\mu+
{N^\mu}_{\beta}dx^\beta
\right)
=\frac{\partial^2 L}{\partial dx^\alpha \partial dx^\mu}
\left( d^2x^\mu+2G^\mu \right). 
\end{eqnarray}
Using the previous non-linear Finsler connection,
the Euler-Lagrange equation of $L$ is equivalent to the
auto-parallel equation
\begin{eqnarray}
 \left\{
 \begin{array}{l}
  d^2x^\mu+2G^\mu(x,dx)=\lambda^0 \ell^\mu_0+\lambda^I \ell^\mu_I,
 \\
 {\cal C}_I:=M_I-L_{Ia}L^{ab}M_b=0,\quad (I=1,2,\dots,D),
 \end{array}
 \right. \label{autop}
\end{eqnarray}
where $\lambda^0, \, \lambda^I$ are arbitrary
$2$-nd homogeneous function with respect to $dx^\mu$.
From physical viewpoint,
(\ref{autop}) are equations of a motion of system constrained on a surface
which is defined by second equations ${\cal C}_I=0$.
The arbitrary function $\lambda^0$ is
determined by taking a time parameter, and another arbitrary functions of $\lambda^I$
are determined from the consistency with derivatives of the second
constraint equations of (\ref{autop}), and
the others of $\lambda^I$ remain arbitrary.
Also in a Riemannian space, we can define {\it Finsler arc length parameter} $s$
which satisfies $L\left(x(s),\frac{dx(s)}{ds}\right)=1$.
Taking the differentiation with respect to $s$
we get a ``time fixing condition'',
\begin{eqnarray}
 0=\bib{L^\ast}{x^\mu}\frac{dx^\mu}{ds}+\bib{L^\ast}{dx^\mu}\frac{d^2 x^\mu}{ds^2},
 \quad L^\ast:=L\left(x(s),\frac{dx(s)}{ds}\right). 
\end{eqnarray}
Using the parameterised auto-parallel equation (\ref{autop}),
\begin{eqnarray}
 \frac{d^2 x^\mu}{ds^2}+2G^\mu\left(x(s),\frac{dx(s)}{ds}\right)
 =\xi^0 \frac{dx^\mu}{ds}+\xi^I v^\mu_I\left(x(s),\frac{dx(s)}{ds}\right), 
\end{eqnarray}
where $\xi^0:={\lambda^0\left(x(s),\frac{dx(s)}{ds}\right)}/{L^\ast}$ and
$\xi^I:=\lambda^I\left(x(s),\frac{dx(s)}{ds}\right)$,
and parameterised property of $G^\mu$ of non-linear Finsler connection (\ref{omega}),
\begin{eqnarray}
 \bib{L^\ast}{x^\mu}\frac{dx^\mu}{ds}=2\bib{L^\ast}{dx^\mu}
 G^\mu\left(x(s),\frac{dx(s)}{ds}\right), 
\end{eqnarray}
we can show $\xi^0=0$ by the following
\begin{eqnarray}
 0=\bib{L^\ast}{x^\mu}\frac{dx^\mu}{ds}+\bib{L^\ast}{dx^\mu}
 \left\{\xi^0 \frac{dx^\mu}{ds}+\xi^I v^\mu_I-2G^\mu\right\}=\xi^0. 
\end{eqnarray}
Therefore, choosing the Finsler arc length parameter corresponds to $\lambda^0=0$.

\section{Examples}

\subsection*{Riemannian case}

Riemannian manifold $(M^{n+1},g)$ can be considered as a Finsler manifold
given by,
\begin{eqnarray}
 L=\sqrt{g_{\mu\nu}(x)dx^\mu dx^\nu}, \quad (\mu,\nu=0,1,2,\dots,n),
\end{eqnarray}
where $g_{\mu\nu}(x)$ are functions of only coordinates $(x^\mu)$ of $M$,
and $L$ is a regular.
We can easily recognize that the Levi-Civita connection 
$\stackrel{{\tiny \mbox{\rm LC}}}{\varGamma^\mu}_{\alpha\beta}$
becomes our ``non-linear'' connection,
ckecking it to satisfy the equation (\ref{Fpreserve}),
\begin{eqnarray}
 \frac{1}{2L} \bib{g_{\mu\nu}}{x^\beta}dx^\mu dx^\nu
 =\frac{g_{\mu\nu}dx^\nu}{L}
 \stackrel{{\tiny \mbox{\rm LC}}}{\varGamma^\mu}_{\alpha\beta} dx^\alpha, 
\end{eqnarray}
and using uniqueness theorem,
Proposition \ref{unique}.

However we will look for it using existence theorem, Proposition \ref{exist}.
i.e. we will directly calculate
\begin{eqnarray}
 2G^\mu=\left(dx^\beta \bib{L}{x^\beta}\right)\frac{dx^\mu}{L}
 +L^{ab}\ell^\mu_a \left(-\bib{L}{x^b}+dx^\rho \bbib{L}{dx^n}{x^\rho}\right),
\end{eqnarray}
where the Greek indices run as $\beta,\mu,\rho=0,1,2,\dots,n$ and 
the Latin indices run as $a,b=1,2,3,\dots, n$, and we use the summation convention
for appearing same labels.
\begin{eqnarray}
p_\mu
=\bib{L}{dx^\mu}
=\frac{dx_\mu}{L},
\quad
L_{\mu\nu}
=\frac{\partial^2 L}{\partial dx^\mu \partial dx^\nu}
=\frac{1}{L}\left\{
g_{\mu\nu}-\frac{dx_\mu dx_\nu}{L^2}
\right\}, 
\end{eqnarray}
and
$(n\times n)$ matrices
$(L_{ab})=\frac{1}{L}\left( g_{ab}-\frac{dx_a dx_b}{L^2}\right)$
has an inverse matrices 
\begin{eqnarray}
(L^{ab})=L\left(
g^{ab}-g^{0a}\frac{dx^b}{dx^0}-g^{0b}\frac{dx^a}{dx^0}+g^{00}\frac{dx^a dx^b}{(dx^0)^2}
\right),
\end{eqnarray}
where $dx_\mu:=g_{\mu\nu}dx^\nu$ and $(g^{\mu\nu})$ is an inverse matrices of $(g_{\mu\nu})$.
Then
\begin{eqnarray*}
&&
dx^\beta \bib{L}{x^\beta}=\frac1{2L}\bib{g_{\mu\nu}}{x^\beta}dx^\beta dx^\mu dx^\nu,
\quad 
\ell^\mu_a=\delta^\mu_a-\frac{dx_a dx^\mu}{L^2}, \quad
L^{ab}\ell^\mu_a=L\left(g^{b\mu}-g^{0\mu}\frac{dx^b}{dx^0}\right),
\\
&&
-\bib{L}{x^j}+dx^\rho \frac{\partial^2 L}{\partial dx^j \partial x^\rho}
=-\frac{1}{2L}\bib{g_{\mu\nu}}{x^j}dx^\mu dx^\nu
+\frac{1}{L}\bib{g_{j\nu}}{x^\rho}dx^\nu dx^\rho
-\frac{1}{2L^3}
\bib{g_{\mu\nu}}{x^\rho}g_{j\sigma}
dx^\mu dx^\nu dx^\rho dx^\sigma,
\end{eqnarray*}
therefore we get
\begin{eqnarray}
2G^\mu
&=&\left(
-\frac12 g^{\mu\nu} \bib{g_{\alpha\beta}}{x^\nu}
+g^{\mu\nu}\bib{g_{\nu\alpha}}{x^\beta}\right)dx^\alpha dx^\beta
=\frac12 g^{\mu\nu}
\left( 
\bib{g_{\nu\alpha}}{x^\beta}+\bib{g_{\nu\beta}}{x^\alpha}
-\bib{g_{\alpha\beta}}{x^\nu}
\right)dx^\alpha dx^\beta \nonumber \\
&=&
\stackrel{{\tiny \mbox{\rm LC}}}{\varGamma^\mu}_{\alpha \beta}dx^\alpha dx^\beta.
\end{eqnarray}
We will add a little comment on $m$-th polynomial form metric such as
\begin{eqnarray}
 L=\sqrt[\leftroot{0} \uproot{2} m] {g_{\mu_1\mu_2\dots,\mu_m}(x)
dx^{\mu_1}dx^{\mu_2}\cdots dx^{\mu_m}}.
\end{eqnarray}
where $m=4$ quartic form metric had been mentioned by Riemann.
In usual line element treatment, we start from 
$g_{\mu\nu}(x,y)=\frac12\bbib{L(x,y)}{y^\mu}{y^\nu}$,
but $g_{\mu\nu}(x,y)$ becomes uglier form if $m\neq 2$.
On the other hand, our non-linear Finsler definitions are
\begin{eqnarray}
 \frac{1}{mL^{m-1}}
 \left(
 \bib{g_{\mu_1\mu_2\dots \mu_m}}{x^\alpha}dx^{\mu_1} dx^{\mu_2}\cdots dx^{\mu_m}
 \right)=
 \frac{g_{\mu\mu_2\mu_3\dots \mu_m}dx^{\mu_2}dx^{\mu_3}\cdots dx^{\mu_m}}{L^{m-1}}
 {N^\mu}_\alpha,
\end{eqnarray}  
therefore they would be quite simple and convenient definition.

\subsection*{Regular simple case (potential system)}

For application to classical dynamics, 
the potential system
which is a particle motion in three dimensional
Euclidian space $\mathbb{R}^3$ influenced by potential force
is the most simple and impotant case.
It's Finsler metric is given by
\begin{eqnarray}
 L=\frac{m}2 \frac{(dx^1)^2+(dx^2)^2+(dx^3)^2}{dx^0}
 -V(x^1)dx^0. 
\end{eqnarray}
We will calculate the $G^\mu$ from this Finsler metric by
\begin{eqnarray}
2G^\mu=\left(dx^\beta \bib{L}{x^\beta} \right)
\frac{dx^\mu}{L}+L^{ab}\ell^\mu_a
\left(
-\bib{L}{x^b}+dx^\rho \frac{\partial^2 L}{\partial dx^b \partial x^\rho}
\right), \label{ex1G}
\end{eqnarray}
where the Greek indices run as $\beta, \mu, \rho = 0,1,2,3$ and
the Latin indices run as $a,b=1,2,3$, and
we also use the summation convention for appearing same labels.
For this Finsler metric is regular, there is no $\ell^\mu_I$ terms
and no constraint equation in (\ref{ex1G}),
and its auto-parallel equation becomes
\begin{eqnarray}
 d^2 x^\mu+2G^\mu(x,dx)=\lambda \ell^\mu_0, \label{autop1}
\end{eqnarray}
with an arbitrary function $\lambda(x,dx)$.
Let us calculate and check this.
\begin{eqnarray}
&&
p_0=-\left\{
\frac{m}{2}\sum_{i=1}^3 \left(\frac{dx^i}{dx^0}\right)^2+V(x^1,x^2,x^3)
\right\}, \quad
p_i=m\frac{dx^i}{dx^0}, \quad (i=1,2,3), \nonumber
\\
&&
(L_{\mu\nu})=
\left(
\begin{array}{cccc}
m\frac{(dx^1)^2+(dx^2)^2+(dx^3)^2}{(dx^0)^3} & -\frac{mdx^1}{(dx^0)^2} &
-\frac{mdx^2}{(dx^0)^2} & -\frac{mdx^3}{(dx^0)^2} \\
-\frac{mdx^1}{(dx^0)^2} & \frac{m}{dx^0} & 0 & 0 \\
-\frac{mdx^2}{(dx^0)^2} & 0 & \frac{m}{dx^0} & 0 \\
-\frac{mdx^3}{(dx^0)^2} & 0 & 0 & \frac{m}{dx^0} \\
\end{array}
\right),
\quad
(L^{ab})=
\left(
\begin{array}{ccc}
 \frac{dx^0}{m} & 0 & 0 \\
 0 & \frac{dx^0}{m} & 0 \\
 0 & 0 & \frac{dx^0}{m} \\
\end{array}
\right),
\nonumber
\\
&&
\ell^\mu_a=L\bib{}{dx^a}\left(\frac{dx^\mu}{L}\right)
=\delta^\mu_a-\frac{mdx^\mu dx^a}{Ldx^0}, \quad
L^{ab}\ell^\mu_a=
 \frac{dx^0 \delta^\mu_b}{m}-\frac{dx^\mu dx^b}{L},
\quad (a,b=1,2,3),
\nonumber
\\
&&
dx^\beta \bib{L}{x^\beta}=- \bib{V}{x^a}dx^0 dx^a, \quad
-\bib{L}{x^b}+dx^\rho \frac{\partial^2 L}{\partial dx^b \partial x^\rho}
=\bib{V}{x^b}dx^0, \quad (a,b=1,2,3),
\nonumber
\\
&&
\hspace{80pt}
2G^\mu=\left\{
\begin{array}{ll}
 \medskip
 \displaystyle
 -2 \bib{V}{x^b}\frac{(dx^0)^2dx^b}{L} & (\mu=0), \\
 \displaystyle
 \bib{V}{x^b}\left\{
 \frac{(dx^0)^2\delta^{ab}}{m}-2\frac{dx^0dx^a dx^b}{L}
 \right\} & (\mu=a=1,2,3).
\end{array}
\right.
\end{eqnarray}
If we take Finsler arc length parameter $s$,
which is defined by
\begin{eqnarray}
 1=:L\left(x^\mu(s),\frac{dx^\mu(s)}{ds}\right)
 =\frac{m}{2}\sum_{a=1}^3 \frac{(\dot{x}^a)^2}{\dot{x}^0}-V(x(s))\dot{x}^0, \label{ex1.L}
 \quad \dot{x}^\mu:=\frac{dx^\mu(s)}{ds}, \label{parameterp1}
\end{eqnarray}
and its derivative by $s$,
\begin{eqnarray}
 m\left(\frac{\dot{x}^a}{\dot{x}^0}\right)\ddot{x}^a
 -\left\{\frac{m}{2}\left(\frac{\dot{x}^a}{\dot{x}^0}\right)^2+V \right\}
 \ddot{x}^0
 -\dot{x}^0\dot{x}^a \bib{V}{x^a}=0, \label{parameterp2}
\end{eqnarray}
are the time gauge fixing conditions. Using this parameter $s$,
\begin{eqnarray}
 2G^0=-2\bib{V}{x^b}(\dot{x}^0)^2\dot{x}^b (ds)^2, \quad
 2G^a=\bib{V}{x^b} \left\{\frac{(\dot{x}^0)^2 \delta^{ab}}{m}
-2\dot{x}^0 \dot{x}^a \dot{x}^b\right\}(ds)^2, 
\end{eqnarray}
and the auto-parallel equation (\ref{autop1}) becomes
\begin{eqnarray}
 \ddot{x}^0-2(\dot{x}^0)^2\dot{x}^b \bib{V}{x^b}
 =\xi \dot{x}^0, \quad
 \ddot{x}^a+\frac{(\dot{x}^0)^2}{m}\bib{V}{x^a}
 -2\dot{x}^0 \dot{x}^a \dot{x}^b \bib{V}{x^b}=\xi \dot{x}^a , \label{autop2}
\end{eqnarray}
where we substitute $\lambda$ to $\xi=\lambda\left(x(s),\frac{dx(s)}{ds}\right)$
which is an arbitrary function of $s$.
By this equation and time gauge fixing conditions
(\ref{parameterp1}) and (\ref{parameterp2}), we can eliminate $\ddot{x}^\mu$ and
then we can get $\xi=0$, which was also showed in previous section in the case
of using a Finsler arc length parameter.

In this case, the auto-parallel equation (\ref{autop2}) with $\xi=0$ can be
also derived from time gauge fixing (\ref{parameterp1}), (\ref{parameterp2}) and
Euler-Lagrange equation of $L\left(x(s),\dot{x}(s)\right)$,
\begin{eqnarray}
 0=\frac{d}{ds} \left\{ \frac{m}{2} \left(\frac{\dot{x}^a}{\dot{x}^0}\right)^2
 +V \right\}, \quad
 0 =\dot{x}^0 \dot{x}^a \bib{V}{x^a}
 -\frac{d}{ds}\left(
 \frac{m \ddot{x}^a}{\dot{x}^0}\right), 
\end{eqnarray}
for this Lagrange system is not a gauge system.

If we choose other parametrisation $t=x^0$ and redefine
$\dot{x}^\mu:=\frac{dx^\mu(t)}{dt}$ and $\ddot{x}^\mu:=\frac{d^2 x^\mu(t)}{dt^2}$, then
\begin{eqnarray}
 2G^0\left(x(t),\frac{dx(t)}{dt}\right)=-\frac{2\dot{x}^a}{L}\bib{V}{x^a}, \quad
 2G^a\left(x(t),\frac{dx(t)}{dt}\right)=\left\{
 \frac{\delta^{ab}}{m}-2\frac{\dot{x}^a \dot{x}^b}{L}
 \right\}\bib{V}{x^b}, 
\end{eqnarray}
and the auto-parallel equation (\ref{autop1}) becomes
\begin{eqnarray}
 0-\frac{2\dot{x}^a}{L}\bib{V}{x^a}=\xi, \quad
 \ddot{x}^a+\frac{1}{m}\bib{V}{x^a}-2\frac{\dot{x}^a \dot{x}^b}{L}\bib{V}{x^b}=\xi \dot{x}^a.
\end{eqnarray}
Therefore the equation corresponds to
$\ddot{x}^a+\frac{1}{m}\bib{V}{x^a}=0$, that is, the usual form of equation of motion.

\subsection*{Constrained system (2nd class constraint)}

In physics, we call a system which has a singular Finsler Lagrangian {\it a gauge system}
or {\it a constrained system}.
Let us consider a specific example of these systems given by
\begin{eqnarray}
 M=\mathbb{R}^3, \quad
 L(x,dx)=x^1 dx^2-x^2 dx^1+\left\{(x^1)^2+(x^2)^2\right\}dx^0.
\end{eqnarray}
Then conjugate momenta $p_\mu$ and $(L_{\mu\nu})$ are
\begin{eqnarray}
 p_0=(x^1)^2+(x^2)^2, \quad p_1=-x^2, \quad p_2=x^1, \quad
 (L_{\mu\nu})=O,
\end{eqnarray}
therefore there are two zero eigenvectors of $(L_{\mu\nu})$,
$v_1^\mu=\delta_1^\mu$ and $v_2^\mu=\delta_2^\mu$ except for $dx^\mu$.
We can get $G^\mu$ from the formula (\ref{solG}),
\begin{eqnarray}
 2G^\mu=2G^\mu_\ast+\lambda^1 v^\mu_1+\lambda^2 v_2^\mu,
 \quad
 2G^\mu_\ast
 :=\frac{2(x^1 dx^1+x^2 dx^2)dx^0 dx^\mu}{L},
\end{eqnarray}
where $\lambda^I \, (I=1,2)$ are arbitrary functions of $x$ and $dx$, and
there are two constraints:
\begin{eqnarray}
 {\cal C}_1=dx^2+x^1 dx^0=0, \quad {\cal C}_2=dx^1-x^2 dx^0=0.
 \label{ex2-const}
\end{eqnarray}
The auto-parallel equation becomes
\begin{eqnarray}
 \left\{
 \begin{array}{l}
 \medskip
 \displaystyle
 d^2 x^0+\frac{2(x^1 dx^1+x^2 dx^2)(dx^0)^2}{L}
 =\lambda dx^0, \\
 \medskip
 \displaystyle
 d^2 x^1+\frac{2(x^1 dx^1+x^2 dx^2)dx^0 dx^1}{L}
 +\lambda^1 =\lambda dx^1, \\
 \displaystyle
 d^2 x^2+\frac{2(x^1 dx^1+x^2 dx^2)dx^0 dx^2}{L}
 +\lambda^2 =\lambda dx^2.
 \label{ex2-autop1}
 \end{array}
 \right.
\end{eqnarray}
We will consider the above equation as a flow on the tangent bundle $TM$.
Using adopted coordinates $(x^\mu,dx^\mu)$ of $TM$,
the generator can be written by
\begin{eqnarray}
 X_T:=dx^\mu \bib{}{x^\mu}-2G^0_\ast\bib{}{dx^0}-2G^0_\ast\frac{dx^i}{dx^0}\bib{}{dx^i}
 +\lambda dx^\mu\bib{}{dx^\mu}-\lambda^1\bib{}{dx^1}-\lambda^2\bib{}{dx^2},
\end{eqnarray}
and the auto-parallel equation (\ref{ex2-autop1}) can be covariantly expressed by
\begin{eqnarray}
 dx^\mu=X_T(x^\mu), \quad d^2 x^\mu=d(dx^\mu)=X_T(dx^\mu). 
\end{eqnarray}
The consistency condition of the flow $X_T$ with constraints (\ref{ex2-const}) is
that they are conserved along the flow:
\begin{eqnarray}
 X_T({\cal C}_1)=X_T({\cal C}_2)=0, \quad {\rm mod}.({\cal C}_1, {\cal C}_2).
 \label{ex2-consistency}
\end{eqnarray}
If we take $\lambda^1$, $\lambda^2$ as
\begin{eqnarray}
 \left\{
 \begin{array}{l}
 \medskip
 \displaystyle
 \lambda^1=\lambda^1_\ast:=
 \lambda(dx^1-x^2dx^0)-\frac{2G^0_\ast}{dx^0}(dx^1-x^2 dx^0)-dx^0dx^2,
 \\
 \displaystyle
 \lambda^2=\lambda^2_\ast:=
 \lambda(dx^2+x^1dx^0)-\frac{2G^0_\ast}{dx^0}(dx^2+x^1 dx^0)+dx^0dx^1,
 \end{array}
 \right.
\end{eqnarray}
then (\ref{ex2-consistency}) are not only satisfied but also the equalities of
(\ref{ex2-consistency}) are exact: that is, they are {\it strong equalities}.
The following covariant equation
\begin{eqnarray}
 \left\{
 \begin{array}{l}
 \medskip
 dx^\mu=X_{T\ast}(x^\mu), \quad d^2x^\mu=X_{T\ast}(dx^\mu),\\
 \displaystyle
 X_{T\ast}
 =dx^\mu \bib{}{x^\mu}-2G^0_\ast\frac{dx^\mu}{dx^0}\bib{}{dx^\mu}
 +\lambda dx^\mu\bib{}{dx^\mu}-\lambda^1_\ast\bib{}{dx^1}-\lambda^2_\ast\bib{}{dx^2},
 \end{array}
 \right. \label{ex2-covhamilton}
\end{eqnarray}
is equivalent to the Hamilton formulation using {\it Dirac bracket}.
Because if we take a time parameter $t:=x^0$, that is,
we divide the former equation of (\ref{ex2-covhamilton}) by $dx^0$ and
the latter by $(dx^0)^2$ and put $\frac{dx^0}{dx^0}=1, \, \frac{d^2 x^0}{(dx^0)}=0$,
then $\lambda$ is determined as $\lambda=2G^0_\ast/dx^0$, and
the equation becomes
\begin{eqnarray}
 \frac{dx^i}{dt}=X(x^i), \quad \frac{dy^i}{dt}=X(y^i), \quad
 X:=\bib{}{x^0}+y^i\bib{}{x^i}+y^2 \bib{}{y^1}-y^1\bib{}{y^2},
 \label{ex2-hamilton}
\end{eqnarray}
where we change the homogenous coordinate $(dx^\mu)$ to the usual $(y^i)$
defined by $y^i=\frac{dx^i}{dx^0}$.
This vector field $X$ of (\ref{ex2-hamilton}) can be given by $X=\bib{}{t}+w(dH,\cdot)$,
which is defined by a following Poisson structure $w$ and Hamiltonian $H$;
\begin{eqnarray}
 w=\bib{}{y^1}\w \bib{}{x^1}+\bib{}{y^2}\w \bib{}{x^2}-\bib{}{y^1}\w \bib{}{y^2},
 \quad H=\frac12\left\{\left(y^1\right)^2+\left(y^2\right)^2\right\}.
\end{eqnarray}
So the equation of this example is equivalent to
\begin{eqnarray}
 \frac{dx^i}{dt}=y^i, \quad \frac{dy^1}{dt}=y^2, \quad \frac{dy^2}{dt}=-y^1.
\end{eqnarray}
This system is a Fermi particle model whose dynamics corresponds harmonic oscillator,
and is also a gauge constraint system having {\it 2nd class constraint}
of Dirac's classification~\cite{Dirac_Yeshiba, Sugano-Kamo}.
As our formulation is covariant (reparamerisation invariant),
the equation (\ref{ex2-covhamilton}) is a covariant generalisation
of Dirac proceduree using our non-linear Finsler connection.

\subsection*{Constrained system (Frenkel's model)}

The next example of gauge system is a quite pathological example, but
it is known as Dirac conjecture does not hold~\cite{Frenkel, Sugano-Kamo}.
\begin{eqnarray}
 M=\mathbb{R}^4, \quad
 L(x,dx)=\frac{dx^2 (dx^3)^2}{(dx^0)^2}-\frac12 x^1 (x^3)^2 dx^0.
\end{eqnarray}
This Euler-Lagrange equation becomes
\begin{eqnarray}
 \left\{
\begin{array}{l}
 \displaystyle
 0=d\left\{ 
 -2\frac{dx^2 (dx^3)^2}{(dx^0)^3}-\frac12 x^1(x^3)^2 \right\}, \\
 \displaystyle
 0=\frac12 (x^3)^2 dx^0, \\
 \displaystyle
 0=d\left(\frac{dx^3}{dx^0}\right)^2, \\
 \displaystyle
 0=x^1 x^3 dx^0+2d\left\{ \frac{dx^2 dx^3}{(dx^0)^2}\right\},
\end{array}
 \right.
 \Leftrightarrow \quad
 \left\{
\begin{array}{l}
 \displaystyle
 0=\frac{dx^3}{dx^0} \,d\left(\frac{dx^3}{dx^0}\right), \\
 \displaystyle
 0=x^1 x^3 dx^0+2d\left\{ \frac{dx^2 dx^3}{(dx^0)^2}\right\},\\
 0=x^3.
\end{array}
 \right. 
\end{eqnarray}
If we take a conventional time parameter $t=x^0$, then
\begin{eqnarray}
 \dot{x}^3 \ddot{x}^3=0, \quad
 2\ddot{x}^2 \dot{x}^3+2\dot{x}^2 \ddot{x}^3+x^1 x^3=0, \quad x^3=0, 
\end{eqnarray}
and these equal to
\begin{eqnarray}
 x^1=\xi^1(t), \quad x^2=\xi^2(t), \quad x^3=0, \label{Fsol}
\end{eqnarray}
where $\xi^1$ and $\xi^2$ are arbitrary function of $t$.

We will express this system to auto-parallel form
using our non-linear Finsler connection.
Conjugate momentum $p_\mu=\bib{L}{dx^\mu}$ are
\begin{eqnarray}
  p_0=
  -2\frac{dx^2 (dx^3)^2}{(dx^0)^3}-\frac12 x^1 (x^3)^2, \quad
  p_1=0, \quad
  p_2=\left(\frac{dx^3}{dx^0}\right)^2,\quad
  p_3=2\frac{dx^2 dx^3}{(dx^0)^2}, 
\end{eqnarray}
and $(L_{\mu\nu})$ is
\begin{eqnarray}
 (L_{\mu\nu})=\left(
 \begin{array}{cccc}
  6\frac{dx^2 (dx^3)^2}{(dx^0)^4} & 0 & -2 \frac{(dx^3)^2}{(dx^0)^3}
   & -4\frac{dx^2 dx^3}{(dx^0)^3} \\
  0 & 0 & 0 & 0 \\
  -2\frac{(dx^3)^2}{(dx^0)^3} & 0 & 0 & 2\frac{dx^3}{(dx^0)^2} \\
  -4\frac{dx^2 dx^3}{(dx^0)^3} & 0 & 2\frac{dx^3}{(dx^0)^2}  & 2\frac{dx^2}{(dx^0)^2}
 \end{array}
 \right). 
\end{eqnarray}
We can recognise ${\rm rank}(L_{\mu\nu})=2$ and so there is one zero eigenvector $v^\mu$
of $(L_{\mu\nu})$ except for $dx^\mu$, and we can take $v_1^\mu=\delta^\mu_1$.
Calculations are following
\begin{eqnarray*}
 &&\ell^\mu_0=\frac{dx^\mu}{L}, \quad
 \ell^\mu_1=v_1^\mu=\delta^\mu_1, \quad
 \ell^\mu_2=\delta^\mu_2-\frac{dx^\mu p_2}{L},\quad
 \ell^\mu_3=\delta^\mu_3-\frac{dx^\mu p_3}{L}, \\
 &&M_0=\frac14(x^3)^2 dx^1+\frac12 x^1 x^3 dx^3, \,
 M_1=\frac14 (x^3)^2dx^0, \, M_2=0, \, M_3=\frac12 x^2 x^3 dx^0, \\
 &&(L^{ab})=\left(
 \begin{array}{cc}
 -\frac{(dx^0)^2 dx^2}{2 (dx^3)^2}  & \frac{(dx^0)^2}{2dx^3} \\
 \frac{(dx^0)^2}{2 dx^3} &  0
 \end{array}
 \right), \,
 (\lambda^a)=(L^{ab}M_b)=\left(
 \begin{array}{c}
 x^2 x^3\frac{(dx^0)^3}{4 dx^3} \\
 0
 \end{array}
 \right), \, (a,b=2,3),\\
 &&dx^\beta \bib{L}{x^\beta}=-\frac12(x^3)^2 dx^0 dx^1-x^1x^3 dx^0 dx^3,\quad
 {\cal C}:=M_1-L^{ab}L_{a1}M_b=\frac14(x^3)^2 dx^0=0.
\end{eqnarray*}
With a constraint ${\cal C}=x^3=0$, $G^\mu$
of the non-linear connection are given by
\begin{eqnarray}
 G^\mu&=&\frac12\left(dx^\beta\bib{L}{x^\beta}\right)\ell^\mu_0
 +\lambda^1 \ell^\mu_1+\lambda^2 \ell^\mu_2+\lambda^3 \ell^\mu_3 \nonumber \\
 &=& \left\{
 -\frac14 (x^3)^2dx^0 dx^1-\frac12 x^1x^3 dx^0 dx^3\right\}\frac{dx^\mu}{L}
 +\lambda^1 \delta^\mu_1+\frac{x^2x^3}{4}\frac{(dx^0)^3}{dx^3}\ell^\mu_2 \nonumber \\
 &=&\lambda^1 \delta^\mu_1
 +\lim_{x^3,dx^3 \to 0} \frac{x^2 x^3}{4}\frac{(dx^0)^3}{dx^3} \delta^\mu_2.
\end{eqnarray}
The last term is ambiguous because the limit cannot be defined.
A corresponding difficulty also occurs in $(L_{ab})$.
The matrix $(L_{ab})$ takes a value with the constraint $x^3=0$
\begin{eqnarray}
 (L_{ab})=\left(
\begin{array}{cc}
 0&0 \\
 0 &\frac{2dx^2}{(dx^0)}
\end{array}
 \right), \quad (a,b=2,3), 
\end{eqnarray}
but it has no inverse.
On a constraint surface defined by ${\cal C}=0$, the total rank of $(L_{\mu\nu})$ becomes one.
In other word, there is also another zero eigenvalued function of $(L_{\mu\nu})$:
$v_2^\mu=\delta^\mu_2$.
Therefore on the constraint surface $x^3=0$, $G^\mu$ has the following form
\begin{eqnarray}
 G^\mu=\frac12\left(dx^\beta \bib{L}{x^\beta}\right)\frac{dx^\mu}{L}
 +\lambda^1 v^\mu_1+\lambda^2 v^\mu_2+\lambda^3 \ell^\mu_3
 =\lambda^1 \delta^\mu_1+\lambda^2 \delta^\mu_2,
\end{eqnarray}
where $\lambda^i, \, (i=1,2)$ are arbitrary function.
Then the auto-parallel equation becomes
\begin{eqnarray}
\left\{
\begin{array}{l}
 d^2 x^0=\lambda dx^0, \\
d^2 x^1=\lambda dx^1+\lambda^1, \\
d^2 x^2=\lambda dx^2+\lambda^2, \\
d^2 x^3=\lambda dx^3. \label{ex3-autop}
\end{array}\right.
\end{eqnarray}
Considering the constraints $x^3=dx^3=0$,
we can rewrite this equation (\ref{ex3-autop}) as Hamilton flow like the previous example,
\begin{eqnarray}
 \left\{
 \begin{array}{l}
 \medskip
 dx^\mu=X_{T\ast}(x^\mu), \quad d^2 x^\mu=X_{T\ast}(dx^\mu),  \\
 \displaystyle
 X_{T\ast}=dx^\mu \bib{}{x^\mu}+\lambda dx^\mu \bib{}{dx^\mu}
 +\xi^1 dx^1 \bib{}{dx^1}+\xi^2 dx^2 \bib{}{dx^2}, \label{ex3-covhamilton}
 \end{array}
 \right.
\end{eqnarray}
where $\lambda^i \, (i=1,2)$ are redefined so as to integrability, and
$\xi^i \, (i=1,2)$ are arbitrary functions.
If we take a time parameter $t=x^0$, then $\lambda=0$ and equation (\ref{ex3-covhamilton})
becomes
\begin{eqnarray}
 \frac{dx^i}{dt}=X(x^i), \quad \frac{dy^i}{dt}=X(y^i), \quad
 X=\bib{}{t}+y^i\bib{}{x^i}+\xi^i(t) y^i \bib{}{y^i},
\end{eqnarray}
which is consistent with the solution of (\ref{Fsol}).
This Frenkel model is a constraint system having {\it 1st class constraint}
of Dirac's classification~\cite{Frenkel, Sugano-Kamo}.
Even in these quite non-trivial example, our non-linear connection
will be convenient for rewriting Lagrangian form to covariant Hamiltonian form
(\ref{ex3-covhamilton}).

\section{Discussions}

We propose a new non-linear Finsler connection which is a generalization of 
the non-linear part of the Berwald's connection.
If Finsler metric $L$ is regular, the Euler-Lagrange equation 
$0=c^\ast \left\{\bib{L}{x^\alpha}-d\left(\bib{L}{dx^\alpha}\right)\right\}$
derived form 
the variational principle of the action functional ${\cal A}[\bm{c}]=\int_{\bm{c}}L$
becomes equivalent to $d^2 x^\mu+2G^\mu(x,dx)=\lambda(x,dx) dx^\mu$ using our
non-linear connection ${N^\mu}_\beta$, $G^\mu=\frac12 {N^\mu}_\beta dx^\beta$,
and arbitrary function $\lambda$.
If we take Finsler arc-length parameter $s$, discussing at section IV,
$\lambda$ becomes zero, and the non-linear connection is given by
${N^\mu}_\beta=\bib{G^\mu}{dx^\beta}$;
that is our ${N^\mu}_\beta$ is exactly the non-linear connection of Berwald.

We can formally define a torsion operator and a curvature operator~\cite{Miron},
\begin{eqnarray*}
 &\displaystyle
 T(X,Y):=\nabla_X Y-\nabla_Y X-[X,Y]=X^\beta Y^\alpha  
 \left\{\bib{{N^\mu}_{\beta}}{dx^\alpha}\left(x,dx(Y)\right)
 -\bib{{N^\mu}_{\alpha}}{dx^\beta}\left(x,dx(X)\right)
 \right\}\bib{}{x^\mu},\\
 & \displaystyle
 {R^\mu}_{\beta\gamma}(x,dx):=\bib{{N^\mu}_\gamma}{x^\beta}-\bib{{N^\mu}_\beta}{x^\gamma}
 +{N^\mu}_{\alpha\beta}{N^\alpha}_\gamma
 -{N^\mu}_{\alpha\gamma}{N^\alpha}_\beta.
\end{eqnarray*}
In spite of symmetry ${N^\mu}_{\alpha\beta}={N^\mu}_{\beta\alpha}$,
the torsion $T(X,Y) \neq 0$ because of non-linearity of ${N^\mu}_{\alpha \beta}$.
Furthermore our treatments bases on point-Finsler viewpoint, and
we only use the non-linear connection and consider tangent vectors on point manifold.
In many cases of physical problems, we can hardly give line element space vector 
$\tilde{X}=\tilde{X}^\mu(x,y)\bib{}{x^\mu}$ proper physical meaning.
We could hope our minimal setting Finsler connection ${N^\mu}_\beta$, its torsion, and
its curvature ${R^\mu}_{\beta\gamma}$ are applicable to fruitful problems in nature.

\begin{acknowledgments}
We thank Lajos Tam\'{a}ssy, Erico Tanaka and Muneyuki Ishida for creative discussions
and careful checking of calculation.
This work was greatly inspired by late Yasutaka Suzuki.
The first author has been supported by the European Union's Seventh
Framework Programme (FP7/2007-2013) under grant agreement no.~317721.
T. Ootsuka thanks JSPS Institutional Program for Young Researcher Overseas Visits.
\end{acknowledgments}

\bibliographystyle{jplain}

\begin{thebibliography}{10}
 
\bibitem{Bao-Chern-Shen}
D.~Bao, S.~S. Chern, and Z.~Shen,
\newblock {\em An Introduction to Riemann-{F}insler Geometry},
\newblock Springer, 2000.

\bibitem{ChernChenLam}
S.~S. Chern, W.~H. Chen, and K.~S. Lam,
\newblock {\em Lectures on Differential Geometry},
\newblock World Scientific, 2000.

\bibitem{Suzuki1956}
Y.~Suzuki,
\newblock Finsler Geometry in Classical Physics,
\newblock {\em Journal of the College of Arts and Sciences, Chiba University}, Vol.~2, pp.
  12--16, 1956.

\bibitem{Lanczos}
C.~Lanczos,
\newblock {\em The Variational Principles of Mechanics},
\newblock Dover Books on Physics, 1986.

\bibitem{Kozma-Tamassy1}
L.~Kozma and L.~Tam\'{a}ssy,
\newblock {F}insler geometry without line elements faced to applications,
\newblock {\em Reports on Mathematical Physics}, Vol.~51, pp. 233--250, 2003.

\bibitem{Aldaya}
V.~Aldaya and J.~A. de~Azc\'arraga,
\newblock Geometric formulation of classical mechanics and field theory,
\newblock {\em Rivista Del Nuovo Cimento}, Vol.~3, pp. 1--66, 1980.

\bibitem{Ootsuka2012}
T.~Ootsuka,
\newblock New covariant {L}agrange formulation for field theories,
\newblock {\em arXiv:1206.6040v1}, 2012.

\bibitem{OYIT2014}
T.~Ootsuka, R.~Yahagi, M.~Ishida, and E.~Tanaka,
\newblock Energy-momentum conservation laws in Finsler/Kawaguchi Lagrangian formulation,
\newblock {\em Classical and  Quantum Gravity}, Vol.~32, 165016, 2015.

\bibitem{Miron}
R.~Miron,
\newblock {\em The Geometry of Higher-Order Lagrange Spaces},
\newblock Kluwer Academic Publishers, 1997.

\bibitem{Dirac_Yeshiba}
P.~A.~M. Dirac,
\newblock {\em Lectures on Quantum Mechanics},
\newblock Yeshiba University, 1964.

\bibitem{Sugano-Kamo}
R.~Sugano and H.~Kamo,
\newblock Poincar\'e-Cartan Invariant Form and Dynamical Systems with
  Constraints,
\newblock {\em Progress of Theoretical Physics}, Vol.~67, pp. 1966--1988, 1982.

\bibitem{Frenkel}
A.~Frenkel,
\newblock Comment on Cawley's counterexample to a conjecture of Dirac,
\newblock {\em Physical Review D}, Vol.~21, No.~10, pp. 2986--2987, 1980.

\end{thebibliography}

\end{document}